\documentclass{amsart}

\usepackage{amssymb, amsmath,mathtools}
\mathtoolsset{showonlyrefs}
\usepackage{mathrsfs}%
\usepackage{amscd}
\usepackage{verbatim}
\usepackage{stmaryrd}
\usepackage[dvipsnames]{xcolor}

\usepackage{enumerate}

\usepackage[colorlinks,linkcolor={blue},citecolor={blue},urlcolor={purple},]{hyperref}

\usepackage[colorinlistoftodos,prependcaption,textsize=tiny]{todonotes}

\usepackage{doi}

\theoremstyle{plain}
\newtheorem{theorem}{Theorem}[section]
\theoremstyle{remark}
\newtheorem{remark}[theorem]{Remark}

\theoremstyle{plain}

\newtheorem{problem}{Problem}
\numberwithin{equation}{section}

\def\R{{\mathbb R}}
\def\C{{\mathbb C}}

\renewcommand\Re{\operatorname{Re}}

\newcommand{\calL}{\mathcal{L}}

\newcommand{\wt}{\widetilde}

\usepackage{stmaryrd}

\renewcommand{\div}{\normalfont{\text{div}}}

\newcommand{\norm}[1]{{\left\vert\kern-0.25ex\left\vert\kern-0.25ex\left\vert #1
    \right\vert\kern-0.25ex\right\vert\kern-0.25ex\right\vert}}

\def\XXint#1#2#3{{\setbox0=\hbox{$#1{#2#3}{\int}$ }
\vcenter{\hbox{$#2#3$ }}\kern-.6\wd0}}

\newcommand{\red}[1]{}

\newcommand{\lb}{\langle}
\newcommand{\rb}{\rangle}

\allowdisplaybreaks

\begin{document}

\date\today

\title[Counterexamples to maximal regularity]{Counterexamples to maximal regularity\\ for operators in divergence form}

\thanks{The first-named author is supported by the Alexander von Humboldt foundation by a Feodor Lynen grant. The second-named author is supported by NSF CAREER Grant DMS-2143668 and a Sloan Research Fellowship. The third-named author is supported by the VICI subsidy VI.C.212.027 of the Netherlands Organisation for Scientific Research (NWO)}

\author{Sebastian Bechtel}
\address{Delft Institute of Applied Mathematics, Delft University of Technology, P.O. Box 5031, 2600 GA Delft, The Netherlands}
\email{S.Bechtel@tudelft.nl}

\author{Connor Mooney}
\address{Department of Mathematics, UC Irvine Rowland Hall 410C, Irvine, CA 92697-3875, USA}
\email{mooneycr@math.uci.edu}

\author{Mark Veraar}
\address{Delft Institute of Applied Mathematics, Delft University of Technology, P.O. Box 5031, 2600 GA Delft, The Netherlands}
\email{M.C.Veraar@tudelft.nl}



\begin{abstract}
In this paper, we present counterexamples to maximal $L^p$-regularity for a parabolic PDE. The example is a second-order operator in divergence form with space and time-dependent coefficients. It is well-known from Lions' theory that such operators admit maximal $L^2$-regularity on $H^{-1}$ under a coercivity condition on the coefficients, and without any regularity conditions in time and space. We show that in general one cannot expect maximal $L^p$-regularity on $H^{-1}(\R^d)$ or $L^2$-regularity on $L^2(\R^d)$.
\end{abstract}

\maketitle

\section{Introduction}

Let $V$ and $H$ be complex Hilbert spaces such that $V\hookrightarrow H$ densely and continuously.
Identifying $H^*$ with its dual, we can view $H$ as a subspace of $V^*$, which is the dual of $V$. We start with the abstract problem
\begin{equation}
\label{eq:abstractMReq}
\tag{CP}
\left\{
\begin{aligned}
u' - \mathcal{A}u &= f, \ \ \text{on} \  (0,1),\\
u(0)&=0.
\end{aligned}\right.
\end{equation}
Here $\mathcal{A}:(0,1)\to \calL(V, V^*)$ is strongly measurable and $f\in L^2(0,1;V^*)$.
Moreover, we suppose that there are $\Lambda, \lambda > 0$ such that for all $t\in (0,1)$ and $v\in V$ one has
\begin{align*}
	\| \mathcal{A}(t) v \|_{V^*} \leq \Lambda \| v \|_V, \qquad \Re \langle \mathcal{A}(t) v, v \rangle \geq \lambda \| v \|_V^2.
\end{align*}
By Lions' theory~\cite{Lions61} it is known that \eqref{eq:abstractMReq} has a unique weak solution $u\in H^1(0,1;V^*)\cap L^2(0,1;V)$.
More generally (see \cite[XVIII.3.5]{DLv5}), if $f\in L^2(0,1;V^*)+L^1(0,1;H)$, then \eqref{eq:abstractMReq} has a unique weak solution $u\in L^2(0,1;V)$ such that $u'\in L^2(0,1;V^*)+L^1(0,1;H)$.

\noindent In this paper we study two problems concerning the regularity of $u$.
\begin{problem}\label{problem1}
Let $p\in (1, \infty)\setminus \{2\}$. Under what condition on $\mathcal{A}$ does the following hold: for all $f\in L^p(0,1;V^*)$ there is a unique weak solution $u$ which is in $L^p(0,1;V)$?
\end{problem}
Using the equation this also gives $u\in H^{1,p}(0,1;V^*)$.
Problem~\ref{problem1} can equivalently be formulated to the question whether there is a constant $C>0$ such that for all step functions $f$ valued in $H$ and with $u$ the unique solution to~\eqref{eq:abstractMReq} given by Lions' result for $p=2$ one has the estimate
\[\|u\|_{L^p(0,1;V)}\leq C \|f\|_{L^p(0,1;V^*)}.\]

It is well-known that regularity results fail for the endpoint cases $p=1$ and $p=\infty$ unless $V = V^*$ and thus $\mathcal{A}$ is a family of bounded operators on $V$ (see \cite{Baillon,Guerre} and also \cite[Thm.~17.4.4 \& Cor.~17.4.5]{HNVW3}). Hence, we can safely concentrate on the case $p\in (1,\infty)$ in this paper.

The second problem is a variation of Lions' problem \cite[p. 68]{Lions61}, who originally asked this question for symmetric $\mathcal{A}$.
\begin{problem}\label{problem2}
Under what condition on $\mathcal{A}$ does the following hold: for all $f\in L^2(0,1;H)$ the unique weak solution $u$ satisfies $u'\in L^2(0,1;H)$?
\end{problem}

Again using the equation one also has $\mathcal{A} u\in L^2(0,1;H)$. However, it is unclear what this tells about the regularity of $u$ since the domain of the operators $\mathcal{A}(t)$ in $H$ are not easy to describe in general.

Both problems have in common (at least if $p>2$) that they imply regularity of $u$ such as $u\in C^{\alpha}([0,1];[H,V]_{\lambda})$ for some $\alpha>0$ and $\lambda\in (0,1)$. Such properties are for instance useful in the study of non-linear problems. They follow from standard interpolation estimates and Sobolev embeddings. In Lions' general $L^2$-setting one can only obtain $C([0,1];H)$ or H\"{o}lder regularity of small order, compare with Theorem~\ref{thm:Sneiberg}.

If $\mathcal{A}$ is autonomous, that is to say, the family $\mathcal{A}(t)$ does not depend on $t$, then in both problems the answer is affirmative. Indeed, in Problem~\ref{problem1} this can be concluded from Lions' result and~\cite[Thm.~17.2.31]{HNVW3}. Concerning Problem~\ref{problem2}, this follows from a result of de Simon~\cite{DeSimon}.

Otherwise, some conditions are needed. In the case of Problem~\ref{problem1} it is sufficient that the mapping $t \mapsto \mathcal{A}(t) \in \calL(V, V^*)$ is (piecewise relatively) continuous \cite{ACFP, PS01}. Without any continuity it is only known that there is some $\varepsilon > 0$ depending on $\Lambda$, $\lambda$ such that Problem~\ref{problem1} holds with $p \in (2-\varepsilon,2+\varepsilon)$, see Theorem~\ref{thm:Sneiberg}. For Problem~\ref{problem2} the situation is even more delicate. Based on an abstract counterexample for the Kato square root problem due to McIntosh, Dier constructed a first nonsymmetric counterexample to Problem~\ref{problem2} in his PhD thesis, see also~\cite[Sec.~5]{Lions-Survey}. Moreover, Fackler constructed a counterexample that is symmetric and $C^\frac{1}{2}$-H\"{o}lder continuous~\cite{Fackler16}. To the contrary, with slightly more regularity (for instance $C^{\frac{1}{2}+\varepsilon}$-H\"{o}lder regularity), many positive results were given~\cite{AO19,AE16,B, DZ17,Fackler18,HO15,PS01}. We also recommend the survey~\cite{Lions-Survey} for an overview of Problem~\ref{problem2}.

The counterexamples due to Dier and Fackler are abstract and not differential operators. As is highlighted by the solution to the Kato square root problem~\cite{Kato}, the extra structure of a differential operator can be beneficial compared to the general situation. It is hence of interest to find counterexamples to Problems~\ref{problem1} and~\ref{problem2} that are differential operators. For Problem~\ref{problem2} this was explicitly pointed out in \cite[Problem 6.1]{Fackler18} and \cite[Prop.~12.1]{Lions-Survey}. To be more precise with our setting, we work with $H = L^2(\R^d)$ and $V = H^1(\R^d)$. If $B \colon [0,1] \times \R^d \to \C^{d\times d}$ is elliptic in the sense that there exist constants $\Lambda, \lambda > 0$ such that for all $t \in [0,1]$ and $x\in \R^d$ one has
\begin{align}
	|B(t,x) \xi| \leq \Lambda |\xi|, \qquad \Re B(x,t) \xi \cdot \overline{\xi} \geq \lambda |\xi|^2 \qquad (\xi \in \C^d),
\end{align}
then consider the problem
\begin{equation}
	\tag{P}
	\label{eq:elliptic_eq}
	\left\{
	\begin{aligned}
		u' - \div(B\nabla u) &= f, \quad \text{on } (0,1),\\
		u(0)&=0,
	\end{aligned}\right.
\end{equation}
where $f \in L^2(0,1; H^{-1}(\R^d))$. In the notation of~\eqref{eq:abstractMReq} we have put $\mathcal{A}(t)v = -\div(B(t)\nabla v)$.
Lions' theory yields a unique solution $u$ in the regularity class $H^1(0,1; L^2(\R^d)) \cap L^2(0,1; H^1(\R^d))$.
Problems~\ref{problem1} and~\ref{problem2} ask if for all elliptic coefficients one has the regularity $u\in L^p(0,1; H^1(\R^d))$ or $u' \in L^2(0,1; L^2(\R^d))$ when the forcing term $f$ is taken from $L^p(0,1; H^{-1}(\R^d))$ or $L^2(0,1; L^2(\R^d))$.
The answer is negative in both cases and the respective counterexamples will be the content of the present article. More precisely, based on a construction of the second-named author \cite{Mooney2}, we will obtain equations whose solution fails to have higher $L^r(L^s)$-integrability (Theorem~\ref{thm:LrLscounter}). Based on this result, we produce a counterexample to Problem~\ref{problem1} for every $p\neq 2$ in Theorem~\ref{thm:LpHmin1}. In particular, this shows sharpness of Theorem~\ref{thm:Sneiberg} which we mentioned briefly above. Concerning Problem~\ref{problem2}, we will show in Theorem~\ref{thm:L2L2} that it fails in the worst possible way for general elliptic problems in divergence form.

\subsection*{Acknowledgment} The authors thank the anonymous referee for their comments on the manuscript.

\section{Maximal regularity in Hilbert spaces}

\subsection{Known positive results}

\noindent
We present the results regarding Problem~\ref{problem1} that were already discussed in the introduction.

A function $u\in L^2(0,1;V) \cap H^1(0,1; V^*)$ is called a solution to \eqref{eq:abstractMReq} if for all $t\in [0,1]$,
\[u(t) - \int_0^t \mathcal{A}u ds = \int_0^t fds,\]
where the integrals are defined as Bochner integrals in $V^*$.

One can check that $u$ is a solution to \eqref{eq:abstractMReq} if and only if for all $\phi\in C^\infty_c((0,1);V)$ one has
\begin{align}\label{eq:weakabstract}
-\int_0^1 \lb u(s), \phi'(s)\rb ds - \int_0^1 \lb \mathcal{A}(s)u(s), \phi(s)\rb  ds  = \int_0^1 \lb f(s), \phi(s)\rb ds.
\end{align}

\begin{theorem}[Lions]\label{thm:Lions}
	In the situation of~\eqref{eq:abstractMReq} there is a unique solution $u\in L^2(0,1;V) \cap H^1(0,1; V^*)$. Constants in the maximal regularity estimate depend only on $\Lambda$ and $\lambda$.
\end{theorem}

Based on a perturbation principle for isomorphisms in complex interpolation scales due to Sneiberg~\cite{Sneiberg-Original}, Lions' result can be extended to $p$ close to $2$, see~\cite[Thm.~4.2]{DiElRe}.

\begin{theorem}\label{thm:Sneiberg}
In the situation of Lions' result (Theorem~\ref{thm:Lions}) there exists $\varepsilon > 0$ depending on $\Lambda$, $\lambda$ and the pair $(V,H)$ such that Problem~\ref{problem1} has a positive solution for $p\in (2-\varepsilon, 2+\varepsilon)$.
\end{theorem}
It will be clear from Theorem~\ref{thm:LpHmin1} that one cannot improve the latter to all $p\in (1, \infty)$, even if $f\in L^p(0,1;H)$.

\subsection{A preliminary counterexample from the literature}
\label{Subsec:Mooney}

Our proof relies on an example by the second-named author from \cite{Mooney2}. We recall some of the details of that example here for the reader's convenience. In \cite{Mooney2} complex-valued solutions $\zeta: (-\infty,\,1) \times \mathbb{R}^d \rightarrow \mathbb{C}$ to linear, uniformly parabolic PDEs with complex coefficients $B: (-\infty,\,1) \times \mathbb{R}^d \rightarrow \mathbb{C}^{d \times d}$  of the form
\begin{align}\label{parabolic}
	\partial_t\zeta = \text{div}(B(t,\,x)\nabla \zeta)
\end{align}
are constructed, with the following properties. First, (\ref{parabolic}) holds in the sense of distributions. Moreover, $\zeta$ and $B$ are smooth away from $(-\infty,\,1) \times \{0\}$ (hence the equation holds classically away from $\{x = 0\}$), and $\zeta$ is locally Lipschitz on $(-\infty,\,1) \times \mathbb{R}^d$. Second, the solutions are chosen to obey certain scaling symmetries. Solutions of the following form are constructed:
\begin{align}\label{symm1}
	\zeta(t,\,x) = (1-t)^{-\mu/2}e^{-i\log(1-t)/2}w(x/(1-t)^{1/2}),
\end{align}
\begin{align}\label{symm2}
	B(t,\,x) = a(x/(1-t)^{1/2}),
\end{align}
for appropriate choices of parameter $\mu \in \mathbb{R}$, function $w : \mathbb{R}^d \rightarrow \mathbb{C}$, and uniformly elliptic coefficients $a : \mathbb{R}^d \rightarrow \mathbb{C}^{d \times d}$. The parabolic PDE (\ref{parabolic}) is equivalent to the following elliptic PDE for $w$ on $\mathbb{R}^d$:
\begin{align}\label{elliptic}
	\text{div}(a(x)\nabla w) = \frac{1}{2}(iw + \mu w + x \cdot \nabla w).
\end{align}
In \cite{Mooney2} it is shown that for any choice of the parameter $\mu$ satisfying
\begin{align}\label{mubound}
	0 \leq \mu < d/2,
\end{align}
one can find a uniformly elliptic matrix field $a$ on $\mathbb{R}^d$ that is smooth away from $0$, and a solution $w$ to ~\eqref{elliptic} on $\mathbb{R}^d$ that is smooth away from $0$ and locally Lipschitz on $\mathbb{R}^d$, such that the following estimates are satisfied: for all multi-indices $\alpha$ and all $|x| \geq 1$, we have
\begin{align}\label{Est1}
	|\partial^{\alpha}w(x)| \leq C_{\alpha}|x|^{-|\alpha|-\mu},
\end{align}
\begin{align}\label{Est2}
	|\partial^{\alpha}a(x)| \leq C_{\alpha}|x|^{-|\alpha|}.
\end{align}
The desired solutions $\zeta$ and coefficients $B$ for the parabolic problem (\ref{parabolic}) are then obtained through~\eqref{symm1} and~\eqref{symm2}. We stress that the coefficients $a$ satisfy the boundedness and uniform ellipticity conditions
\begin{align}\label{ellipticity}
	|a(x)\xi| \leq \Lambda |\xi|, \quad \text{Re}(a(x)\xi \cdot \overline{\xi}) \geq \lambda |\xi|^2
\end{align}
for some $\lambda,\,\Lambda > 0$ depending on $\mu$ and all $x \in \mathbb{R}^d$ and $\xi \in \mathbb{C}^d$, so $B$ satisfies the conditions (1.1) with the same $\lambda$ and $\Lambda$.

We will only use the properties discussed above in the sequel. For the details of how $w$ and $a$ are chosen, the interested reader can consult Section 3 in
\cite{Mooney2}. Roughly speaking, for judicious choices of the form of $w$ and $a$, one can reduce the PDE (\ref{elliptic}) to a system of ODEs. These can be solved by fixing a choice of $w$ and solving for the coefficients in $a$.

\begin{remark}
	When split into real and imaginary parts, the equation (\ref{parabolic}) can be understood as a system of two equations, which are coupled through the imaginary part of $B$. An important philosophical point in \cite{Mooney2} is that when the imaginary part of $B$ is taken to be symmetric, that part of $B$ doesn't contribute to the ellipticity condition. This allows strong coupling of the equations without breaking the ellipticity condition.
\end{remark}

\subsection{Failure of $L^r(L^s)$-integrability for variational solutions}

The following counterexample is based on the construction from the previous subsection and is the basis for our subsequent counterexamples to Problems~\ref{problem1} and~\ref{problem2}.

\begin{theorem}[Failure of higher integrability]
	\label{thm:LrLscounter}
	Let $d \geq 2$. For every $s,r \in (1,\infty)$ with $\frac{2}{r} + \frac{d}{s} < \frac{d}{2}$ there exists an elliptic matrix $B \colon [0,1] \times \R^d \to \C^{d\times d}$ and $f\in L^\infty(0,1; L^2(\R^d))$ such that the unique solution $u$ to~\eqref{eq:elliptic_eq}
	satisfies $u \not\in L^r(0,1; L^s(\R^d))$.
\end{theorem}

\begin{proof}
	Pick $\frac{2}{r} + \frac{d}{s} < \mu < \frac{d}{2}$ and fix corresponding $w$, $a$, $\zeta$, and $B$ as described in Section~\ref{Subsec:Mooney}.

	\noindent\textbf{Step 1}: general setup.
	It is clear from~\eqref{parabolic} and the properties of $w$ and $B$ that the weak derivative $\partial_t \zeta$ exists on $(0,T)$ as an $L^2(B_R)$-valued function, for any $T<1$ and $R>0$, where $B_R$ is a shorthand notation for the Euclidean ball of radius $R$ in $\R^d$.
	Now let
	\begin{align*}
		u(t,x) &= t \zeta(t,x) \eta(x),
		\\ f(t,x) &= \eta(x)\zeta(t,x) -  t \sum_{k,\ell=1}^d [B_{k,\ell}(t,x)+ B_{\ell,k}(t,x)]\partial_k \zeta(t,x)\partial_\ell \eta(x) \\
		&\qquad\qquad\qquad\qquad\qquad-  t \zeta(t,x) \div(B(t,x) \nabla \eta(x)),
	\end{align*}
	where the cut-off function $\eta\in C^\infty_c(\R^d)$ is such that $\eta(x)= 1$ for $|x|\leq 1$, $\eta(x)=0$ for $|x|\geq 2$ and $0\leq \eta\leq 1$ on $\R^d$.

	We claim that $u\in L^2(0,1;H^1(\R^d))\cap H^1(0,1;H^{-1}(\R^d))$, $f\in L^\infty(0,1;L^2(\R^d))$ and $u$ is the unique solution to \eqref{eq:elliptic_eq}. Then, using \eqref{parabolic} it is elementary to check that for $t\in (0,1)$ and $x\in \R^d\setminus\{0\}$ one pointwise has
	\begin{equation*}
		\left\{
		\begin{aligned}
			\partial_t u(t,x) - \div(B(t,x)\nabla u(t,x)) &= f(t,x), \\
			u(t,0)&=0.
		\end{aligned}\right.
	\end{equation*}
Moreover, it also holds in distributional sense on $(0,1)\times \R^d$. Hence, by density and \eqref{eq:weakabstract}, it is a solution to \eqref{eq:elliptic_eq}.

	Therefore, it remains to check $u\in L^2(0,1;H^1(\R^d))$ and $f\in L^\infty(0,1;L^2(\R^d))$ in order to find that $u$ is the unique solution to \eqref{eq:elliptic_eq} provided by Theorem \ref{thm:Lions}.

	\noindent\textbf{Step 2}: $u \in L^2(0,1; H^1(\R^d))$.
	Fix $t\in (0,1)$ and let $\alpha$ be a multiindex with $|\alpha| \leq 1$. By definition of $\zeta$ and using~\eqref{Est1} one has
	\begin{align}
		\label{eq:estuniformu}
		\begin{split}
			\|\partial^\alpha \zeta(t,\cdot)\|_{L^2(B_2)} &\leq  (1-t)^{-\frac{\mu+|\alpha|}{2} + \frac{d}{4}} \|\partial^\alpha w\|_{L^2(B_{{2}/{(1-t)^{1/2}}})}
			\\ &\leq (1-t)^{-\frac{\mu+|\alpha|}{2} + \frac{d}{4}}\Big[C_{d,w} + C_{0} \||\cdot|^{-\mu-|\alpha|}\|_{L^2(B_{{2}/{(1-t)^{1/2}}}\setminus B_1)}\Big]
			\\ & \leq  (1-t)^{-\frac{\mu+|\alpha|}{2} + \frac{d}{4}} \Big[C_{d,w} + C_{0}C_{d} (1-t)^{\frac{\mu+|\alpha|}{2} - \frac{d}{4}} \Big]\\
			&\leq C_{d,w} (1-t)^{-\frac{\mu+|\alpha|}{2} + \frac{d}{4}} + C_{d,w}'
		\end{split}
	\end{align}
	for some constant $C_{d,w}'$ only depending on $d$ and $w$ and with $$C_{d,w} = |B_1|^{1/2} \sup_{x\in B_1} (|w(x)| + |\nabla w(x)|).$$ Note that $C_{d,w}$ is finite since $w$ is Lipschitz. Recall that $\frac{d}{4} - \frac{\mu}{2}$ is positive. Thus, by definition of $u$ and using~\eqref{eq:estuniformu} we find
	\begin{align*}
		\| u(t,\cdot) \|_{L^2(\R^d)} \leq t \|\zeta(t,\cdot)\|_{L^2(B_2)} \leq C_{d,w} + C_{d,w}'.
	\end{align*}
	Thus, $u$ is bounded as an $L^2(\R^d)$-valued function and therefore in particular $u \in L^2(0,1; L^2(\R^d))$. Similarly,
	\begin{align*}
		\|\nabla u(t,\cdot)\|_{L^2(\R^d)} &\leq  \|\nabla \zeta(t,\cdot)\|_{L^2(B_2)} +  \|\nabla \eta\|_{\infty} \|\zeta(t,\cdot)\|_{L^2(B_2\setminus B_1)} \\
		&\leq C_{d,w} (1-t)^{-\frac{\mu+1}{2} + \frac{d}{4}} + C_{d,w}' + C_{d,w}'\|\eta\|_{\infty}.
	\end{align*}
	By choice of $\mu$ we have $-\frac{\mu+1}{2} + \frac{d}{4} > -\frac{1}{2}$, therefore also $\nabla u \in L^2(0,1; L^2(\R^d))$.

	\noindent\textbf{Step 3}: $f\in L^\infty(0,1; L^2(\R^d))$.
	To estimate the norm of $f$ we consider all parts separately. The fact that $\eta \zeta\in L^\infty(0,1;L^2(\R^d))$ follows from \eqref{eq:estuniformu}. Due to the support properties of $\nabla \eta$ and the boundedness of $B$ it suffices to prove a uniform estimate for $\|\partial_m \zeta(t,\cdot)\|_{L^2(B_2\setminus B_1)}$ to estimate $B_{k,\ell}\partial_m \zeta \partial_n \eta$. For this calculate with~\eqref{Est1} that
	\begin{align}
		\label{eq:zeta_bound}
		\begin{split}
		\|\partial_m \zeta(t,\cdot)\|_{L^2(B_2\setminus B_1)} &= (1-t)^{-\frac{\mu}{2} - \frac12} \| (\partial_m w) (\cdot/(1-t)^{1/2} )\|_{L^2(B_2\setminus B_1)}
		\\
		&\leq C_{1} \||\cdot|^{-1-\mu} \|_{L^2(B_2\setminus B_1)}.
		\end{split}
	\end{align}
	It remains to estimate $\zeta \div(B \nabla \eta)$. Again, due to the support properties of $\nabla \eta$, it suffices to consider $x\in B_2\setminus B_1$. First, by \eqref{Est1} and definition of $\zeta$,
	\[|\zeta(t,x)|\leq (1-t)^{-\mu/2} |w(x/(1-t)^{1/2})| \leq C_{1,0}.\]
	Thus it remains to estimate $\|\partial^{\alpha} B_{k,\ell}(t,\cdot) \partial^{\beta} \eta\|_{L^2(B_2\setminus B_1)}$ for $|\alpha|\leq 1$ and $1\leq |\beta| \leq 2$. Since $B_{k,\ell}$ and $\partial^{\beta} \eta$ are uniformly bounded, it is enough to show that $\|\partial^{\alpha} B_{k,\ell}(t,\cdot)\|_{L^2(B_2\setminus B_1)}$ is uniformly bounded for $|\alpha|=1$.
	Indeed, this follows from the analogous calculation to~\eqref{eq:zeta_bound} using~\eqref{Est2} instead of~\eqref{Est1}.

	\noindent\textbf{Step 4}: $u \not\in L^r(0,1; L^s(\R^d))$.
	Similar to~\eqref{eq:estuniformu} but using $B_1\subseteq B_{{1}/{(1-t)^{1/2}}}$,
	\begin{align*}
		\|\zeta(t,\cdot)\|_{L^s(B_1)} = (1-t)^{-\frac{\mu}{2} + \frac{d}{2s}}  \|w\|_{L^s(B_{1/(1-t)^{1/2}})} \geq (1-t)^{-\frac{\mu}{2} + \frac{d}{2s}}  \|w\|_{L^s(B_{1})}.
	\end{align*}
	Therefore, with $t \in (\frac{1}{2},1)$,
	\begin{align*}
		\| u(t,\cdot) \|_{L^s(\R^d)} \geq t \|\zeta(t,\cdot)\|_{L^s(B_1)} \geq \frac{1}{2} (1-t)^{-\frac{\mu}{2} + \frac{d}{2s}}  \|w\|_{L^s(B_{1})}.
	\end{align*}
	By choice of $\mu$, $r$ and $s$ we have $-\frac{\mu}{2} + \frac{d}{2s} < -\frac{1}{r}$, therefore $u \not\in L^r(0,1; L^s(\R^d))$.
\end{proof}

\subsection{Negative result concerning problem \ref{problem1}}

The following result shows that for time-dependent operators in the variational setting, maximal $L^2$-regularity cannot be extrapolated to maximal $L^p$-regularity for $p\neq 2$ besides the small interval given in Theorem \ref{thm:Sneiberg}. It answers Problem \ref{problem1} in a negative way in the setting of elliptic differential operators.

\begin{theorem}[Failure of extrapolation of maximal $L^p$-regularity]\label{thm:LpHmin1}
Let $d\geq 2$. For every $p\in (1, \infty)\setminus\{2\}$
there exists $B:[0,1]\times \R^d\to \C^{d\times d}$ elliptic
and $f \in L^p(0,1;H^{-1}(\R^d))$ such that the unique solution $u$ to~\eqref{eq:elliptic_eq}
satisfies $u\notin L^p(0,1;H^{1}(\R^d))$.
\end{theorem}
\begin{proof}
We divide the proof into two cases.

\textbf{Case 1}: $p > 2$. We appeal to Theorem~\ref{thm:LrLscounter}. If $d \geq 3$ put $r = p$ and $s = 2^* \coloneqq \frac{2d}{d-2}$ and if $d = 2$ put $r = p$ and $s > \frac{2p}{p-2}$. In both cases the condition $\frac{2}{r} + \frac{d}{s} < \frac{d}{2}$ is satisfied, so that Theorem~\ref{thm:LrLscounter} yields $B:[0,1]\times \R^d\to \C^{d\times d}$ elliptic and $f\in L^\infty(0,1; L^2(\R^d)) \subseteq L^p(0,1; L^2(\R^d))$ such that the unique solution $u$ to~\eqref{thm:LrLscounter} satisfies $u \not\in L^r(0,1; L^s(\R^d))$. This implies $u \not\in L^p(0,1; H^1(\R^d))$ since otherwise the Sobolev embedding yields a contradiction to the previous assertion.

\textbf{Case 2}: $p < 2$. We use a duality argument that resembles~\cite[Thm.~6.2]{Krylov07}. By the first part there are $f \in L^{p'}(0,1; H^{-1}(\R^d))$, elliptic coefficients $B \colon [0,1] \times \R^d \to \C^{d\times d}$ and a solution $u \in L^2(0,1; H^1(\R^d)) \cap H^1(0,1; H^{-1}(\R^d))$ to the equation
\begin{equation}
	\label{eq:heat_contra}
	\left\{
	\begin{aligned}
		\partial_t u - \div(B\nabla u) &= f, \\
		u(0)&=0,
	\end{aligned}\right.
\end{equation}
such that $u \not\in L^{p'}(0,1; H^1(\R^d))$. Assume for the sake of contradiction that for every $A \colon [-1,0] \times \R^d \to \C^{d\times d}$ and $\wt{g} \in L^2(-1,0; H^{-1}(\R^d))$ there is a unique solution $\wt{v} \in L^2(-1,0; H^1(\R^d)) \cap H^1(-1,0; H^{-1}(\R^d))$ to the problem
\begin{equation}
	\label{eq:heat_dual}
	\left\{
	\begin{aligned}
		\partial_t \wt{v} - \div(A\nabla \wt{v}) &= \wt{g}, \\
		\wt{v}(-1)&=0,
	\end{aligned}\right.
\end{equation}
satisfying the estimate
\begin{align}
	\label{eq:heat_dual_bound}
	\| \partial_t \wt{v} \|_{L^p(-1,0;H^{-1}(\R^d))} + \| \wt{v} \|_{L^p(-1,0; H^1(\R^d))} \leq C \| \wt{g} \|_{L^p(-1,0; H^{-1}(\R^d))},
\end{align}
where the constant $C>0$ does not depend on $\wt{v}$ and $\wt{g}$. By translation the question on $(-1,0)$ is equivalent to that on $(0,1)$. Both intervals are related by the transformation $v \mapsto -v$. We write for instance $\wt{u}(t) = u(-t)$ to translate $u$ to a function on $(-1,0)$ and vice versa. Specialize $A = (\wt{B})^*$ in~\eqref{eq:heat_dual}. Since $u(0) = 0 = v(1)$ we can use integration by parts to obtain
\begin{align*}
	\int_0^1 \langle g, u \rangle dt = \int_{-1}^0 \langle \wt g, \wt u \rangle dt
	&= \int_{-1}^0 \langle (\wt v)', \wt u \rangle + ((\wt{B})^* \nabla \wt v | \nabla \wt u)_2 dt \\
	&= \int_0^1 - \langle v', u \rangle + (\nabla v | B \nabla u)_2 dt \\
	&= \int_0^1 \langle v, u' \rangle + (\nabla v | B \nabla u)_2 dt.
\end{align*}
Plug in~\eqref{eq:heat_contra} to deduce
\begin{align*}
	\int_0^1 \langle g, u \rangle dt = \int_0^1 \langle v, f \rangle dt.
\end{align*}
Hence, using~\eqref{eq:heat_dual_bound} we can estimate
\begin{align*}
	\bigl| \int_0^1 \langle g, u \rangle dt \bigr| & \leq \| v \|_{L^p(0,1; H^1(\R^d))} \| f \|_{L^{p'}(0,1; H^{-1}(\R^d))}
\\ & \leq C \| g \|_{L^p(0,1; H^{-1}(\R^d))} \| f \|_{L^{p'}(0,1; H^{-1}(\R^d))}.
\end{align*}
Since $g$ was arbitrary, deduce by duality the contradiction $u \in L^{p'}(0,1; H^1(\R^d))$.
\end{proof}

\begin{remark}
For the case $p>2$, we saw that although $f\in L^\infty(0,1;L^2(\R^d))$, the unique solution $u$ to~\eqref{eq:elliptic_eq}
satisfies $u\notin L^p(0,1;H^{1}(\R^d))$.
\end{remark}

\subsection{Negative result concerning problem \ref{problem2}}
We show that Problem~\ref{problem2} fails in the worst possible way in the case of elliptic operators in divergence form: for every $\nu\in (1/2,1]$ there exist coefficients $B$ and a forcing term $f\in L^\infty(0,1;L^2(\R^d))$ such that the unique solution $u$ to~\eqref{eq:elliptic_eq} satisfies $u\notin H^{\nu}(0,1;L^2(\R^d))$. This solves Lions' problem in a negative way also for elliptic operators. It would still be interesting to find a counterexample for $\nu =1$ where the coefficients are H\"older continuous of order $\alpha$ for some $\alpha\leq 1/2$, and this would then be an optimal counterexample due to the positive results for which we refer to the survey \cite{Lions-Survey}. The present coefficient function cannot even be continuous, see Remark~\ref{rem:coefficients}. There seems to be some room for improvement in the regularity of $B$ as the regularity is much worse than $H^1$ in time.

\begin{theorem}[Failure of maximal $L^2$-regularity on $L^2(\R^d)$]\label{thm:L2L2}
Let $d\geq 2$. For every $\nu\in(1/2,1]$
there exists $B:[0,1]\times \R^d\to \C^{d\times d}$ elliptic
and $f \in L^\infty(0,1;L^2(\R^d))$ such that the unique solution $u$ to~\eqref{eq:elliptic_eq}
satisfies $u\notin H^{\nu}(0,1;L^2)$.
\end{theorem}

\begin{proof}
	The strategy is as for the case $p>2$ in Theorem~\ref{thm:LpHmin1}. Let $\nu\in (1/2,1]$ and let $0 < \theta < \frac{1}{2\nu}$. Then define parameters $r$ and $s$ through $\frac{1}{r} = \frac{1}{2} - \nu \theta$ and $\frac{1}{s} = \frac{1}{2} - \frac{1-\theta}{d}$. They satisfy the condition in Theorem~\ref{thm:LpHmin1}, consequently there is $B:[0,1]\times \R^d\to \C^{d\times d}$ elliptic
	and $f \in L^\infty(0,1;L^2(\R^d))$ such that the unique solution $u$ to~\eqref{eq:elliptic_eq} satisfies $u \notin L^r(0,1; L^s(\R^d))$. Suppose that $u\in H^{\nu}(0,1;L^2(\R^d))$. Since $u\in L^2(0,1;H^1(\R^d))$ by Theorem \ref{thm:Lions}, complex interpolation (see \cite[Theorem 14.7.12]{HNVW3}) yields $u\in H^{\nu\theta}(0,1;H^{1-\theta}(\R^d))$. By choice of $\theta$, the Sobolev embedding is applicable and gives $u\in L^r(0,1;L^s(\R^d))$, a contradiction.
\end{proof}

\begin{remark}
	\label{rem:coefficients}
	By construction, the given counterexample is at the same time a counterexample for Problem~\ref{problem1}. We have discussed in the introduction that if $t \mapsto B(t) \in L^\infty(\R^d)$ is (piecewise relatively) continuous, then Problem~\ref{problem1} is automatically true for all $p \in (1, \infty)$. Therefore, the function $B$ in Theorem~\ref{thm:L2L2} is not continuous as a map into $L^\infty(\R^d)$.
\end{remark}

The examples show some limitations of what regularity estimates can hold for elliptic operators in divergence form. However, several issues remain for Problems \ref{problem1} and \ref{problem2}. For instance, the operator $\mathcal{A}$ used in the above counterexamples is not symmetric/hermitian. We do not know what can be said for the case $d=1$. The only reference we found on the one-dimensional setting is \cite{Kry16}, where counterexamples are given in case of $L^p$-integrability in time and space for both the divergence and non-divergence setting for the range $p\in (1, 3/2)\cup (3, \infty)$. Finally, it would be interesting to see what can be said about Problem \ref{problem2} if $\mathcal{A}$ is more regular in time (i.e. continuous or H\"older continuous).

\end{document}